\documentclass[12 pt, leqno]{article}
\newenvironment{proof}{\noindent {\bf Proof:}}{$\Box$ \vspace{2 ex}}

\usepackage{amsmath, amssymb, amscd}
\usepackage{latexsym, epsfig, color, verbatim}

\usepackage{fullpage}

\newcommand{\C}{\mathbb{C}}
\newcommand{\Z}{\mathbb{Z}}

\newcommand{\Q}{\mathbb{Q}}
\newcommand{\R}{\mathbb{R}}

\newcommand{\A}{\mathcal{A}}
\newcommand{\B}{\mathcal{B}}

\newcommand{\ra}{\rightarrow}

\newcommand\Aut{\operatorname{Aut}}

\newcommand\im{\operatorname{image}}
\newcommand\Stab{\operatorname{Stab}}
\newcommand\Gal{\operatorname{Gal}}
\newcommand\Disc{\operatorname{Disc}}
\newcommand\Norm{\operatorname{Norm}}
\newcommand\GL{\operatorname{GL}}

\newtheorem{proposition}{Proposition}[section]
\newtheorem{theorem}[proposition]{Theorem}
\newtheorem{corollary}[proposition]{Corollary}

\newtheorem{question}[proposition]{Question}
\newtheorem{lemma}[proposition]{Lemma}

\newenvironment{notation}{\vspace{2 ex}{\noindent{\bf Notation. }}}{\vspace{2 ex}}
\newenvironment{remark}{\vspace{2 ex}{\noindent{\bf Remark. }}}{\vspace{2 ex}}

\title{Mass formulas for local Galois representations to wreath products and cross products\footnote{The final version of this
paper will be published in \textit{Algebra and Number Theory} }\\
}
    
\author{Melanie Matchett Wood\thanks{email: melanie.wood@math.princeton.edu}\\
Princeton University, Department of Mathematics\\
Fine Hall, Washington Road\\
Princeton, NJ 08544}
\date{\today}

\begin{document}

\maketitle

\begin{abstract}
Bhargava proved a formula for counting, with certain weights, degree $n$ \'{e}tale extensions of a local field, or equivalently, local Galois representations to $S_n$. This formula is motivation for his conjectures about the density of discriminants of $S_n$-number fields. We prove there are analogous ``mass formulas'' that count local Galois representations to any group that can be formed from symmetric groups by wreath products and cross products, corresponding to counting towers and direct sums of \'{e}tale extensions. 
We obtain as a corollary that the above mentioned groups have rational character tables.
Our result implies that $D_4$ has a mass formula for certain weights, but we show that $D_4$ does not have a mass formula when the local Galois representations to $D_4$ are weighted in the same way as representations to $S_4$ are weighted in Bhargava's mass formula.
\end{abstract}



\begin{section}{Introduction}

Bhargava \cite{Bhargava} proved the following mass formula for counting isomorphism classes of \'{e}tale extensions of degree $n$ of a local field $K$: 
\begin{equation}\label{E:BharMF}
\sum_{[L:K]=n\ \mathrm{\acute{e}tale}}\frac{1}{|\Aut(K)|}\cdot \frac{1}{\Norm(\Disc_K L)}=\sum_{k=0}^{n-1} p(k,n-k)q^{-k};
\end{equation}
here $q$ is the cardinality of the residue field of $K$, and $p(k,n-k)$ denotes the number of partitions of $k$ into at most $n-k$ parts.
Equation~\eqref{E:BharMF} is proven using Serre's beautiful mass formula \cite{Serre} which counts
 totally ramified degree $n$ extensions of a local field.  Equation~\eqref{E:BharMF} is at the heart of the conjecture
  \cite[Conjecture 1]{Bhargava} Bhargava makes 
for the asymptotics of the number of $S_n$-number fields with discriminant $\leq X$, and also the conjectures
\cite[Conjectures 2-3]{Bhargava} Bhargava makes for
the relative asymptotics of $S_n$-number fields with certain local behaviors specified.  These conjectures are theorems for $n\leq 5$
(see \cite{DH}, \cite{Bh1}, \cite{Bh2}).

Kedlaya \cite[Section 3]{Kedlaya} has translated Bhargava's formula into the language of Galois representations so that the sum in Equation~\eqref{E:BharMF} becomes a sum over Galois representations to $S_n$ as follows:
 \begin{equation}\label{E:BharMFrep}
\frac{1}{n!}\sum_{\rho : \Gal(K^{\mathrm{sep}}/K) \ra S_n} \frac{1}{q^{c(\rho)}}=\sum_{k=0}^{n-1} p(k,n-k)q^{-k};
\end{equation}
here $c(\rho)$ denotes the Artin conductor of $\rho$ composed with the standard representation $S_n\ra \GL_n(\C)$.  
  
What is remarkable about the mass formulas in Equations~\eqref{E:BharMF} and \eqref{E:BharMFrep} 
is that the right hand side only
depends on $q$ and, in fact, is a polynomial (independent of $q$) evaluated at $q^{-1}$.  
A priori, the left hand sides could depend
on the actual local field $K$, and even if they only depended on $q$, it is not clear there should be a uniform way to write them
as a polynomial function of $q^{-1}$.  This motivates the following definitions.
Given a local field $K$ and a finite group $\Gamma$, let $S_{K,\Gamma}$ denote the set of continuous homomorphisms $\Gal(K^{\mathrm{sep}}/K)\ra \Gamma$ (for the discrete topology on $\Gamma$) and let
$q_K$ denote the size of the residue field of $K$.
Given a function
$c \colon  S_{K,\Gamma} \ra \Z_{\geq 0}$, we define the \emph{total mass} of $(K,\Gamma,c)$
to be
$$
M(K,\Gamma,c):=\sum_{\rho\in S_{K,\Gamma}} \frac{1}{q_K^{c(\rho)}}.
$$
(If the sum diverges, we could say the mass is $\infty$ by convention.
In most interesting cases, e.g. see \cite[Remark 2.3]{Kedlaya},
and all cases we consider in this paper, the sum will be convergent.)
  Kedlaya gave a similar definition, but one should note that our
definition of mass differs from that in \cite{Kedlaya} by a factor of $|\Gamma|$.  
  In \cite{Kedlaya}, $c(\rho)$ is always taken
 to be the Artin conductor of the composition of $\rho$ and some $\Gamma \ra \GL_n(\C)$.  We
 refer to such $c$ as the \emph{counting function attached to the representation} $\Gamma \ra \GL_n(\C)$.  
In this paper, we consider more general $c$.
 
 Given a group $\Gamma$, 
 a \emph{counting function for $\Gamma$} is any function
$c \colon \bigcup_{K} S_{K,\Gamma} \ra \Z_{\geq 0}$ (where the union is over all isomorphism classes of local fields)
such that $c(\rho)=c(\gamma \rho \gamma^{-1})$ for every $\gamma \in \Gamma$. 
 (Since an isomorphism of local fields only determines an isomorphism of their
 absolute Galois groups up to conjugation, we need this condition in order for the counting functions to be
 sensible.)
Let $c$ be a counting function for $\Gamma$ and $S$ be a class of local fields.
We say that $(\Gamma,c)$ has a \emph{mass formula} for $S$ if there exists a polynomial $f(x)\in\Z[x]$ such that for all local fields $K\in S$ we have
$$
M(K,\Gamma,c)=f\left(\frac{1}{q_K}\right).
$$
 We also say that $\Gamma$ has a mass formula for $S$ if there is a $c$ such that $(\Gamma,c)$ has a mass formula for $S$.

  Kedlaya \cite[Theorem 8.5]{Kedlaya} proved that $(W(B_n),c_{B_n})$ has a mass formula for all local fields, where 
  $W(B_n)$ is the Weyl group of $B_n$ and $c_{B_n}$ is the counting function attached to
  the Weyl representation of $B_n$.  This is in analogy with Equation~\eqref{E:BharMFrep} which shows that 
  $(W(A_n),c_{A_n})$ has a mass formula for all local fields, where 
  $W(A_n)\cong S_n$ is the Weyl group of $A_n$ and $c_{A_n}$ is the counting function attached to
  the Weyl representation of $A_n$.
Kedlaya's analogy is very attractive, but he found that it does not extend to the Weyl groups of $D_4$ or $G_2$ when the counting
function is the one attached to the Weyl representation; he showed that mass formulas for all local fields do not exist for those groups and those
particular
counting functions.

The main result of this paper is the following.
\begin{theorem}\label{T:allMF}
Any permutation group that can be constructed from the symmetric groups $S_n$ using wreath products and cross products
has a mass formula for all local fields.
\end{theorem}
Kedlaya's mass formula \cite[Theorem 8.5]{Kedlaya} for $W(B_n)\cong S_2 \wr S_n$   
was the inspiration for this result, and it is now a special case of Theorem~\ref{T:allMF}.

In \cite[Section 8.2]{Bhargava}, Bhargava asks whether his conjecture for $S_n$-extensions about the relative asymptotics of the number of
 global fields
with specified local behaviors holds for other Galois groups.
Ellenberg and Venkatesh \cite[Section 4.2]{Ellenberg} suggest that we can try to count extensions of global fields by quite general invariants of Galois representations.  
In \cite{Abelian}, it is shown that when counting by certain invariants of abelian global fields, 
such as conductor, Bhargava's question can be answered affirmatively.  It is also shown in \cite{Abelian} that when counting abelian global fields by discriminant, 
 the analogous conjectures fail in at least some
cases.
In light of the fact that Bhargava's conjectures for the asymptotics of the number of $S_n$-number fields arise from
his mass formula \eqref{E:BharMF} for counting by discriminant, one naturally looks for mass formulas that use other ways of counting,
such as Theorem~\ref{T:allMF},
which might inspire conjectures for the asymptotics of counting global fields with other Galois groups.

In Section~\ref{S:proof}, we prove that if groups $A$ and $B$ have certain refined mass formulas, 
then $A\wr B$ and $A\times B$ also have such refined mass formulas, which inductively proves Theorem~\ref{T:allMF}.  
Bhargava's  mass formula for $S_n$, given in Equation~\eqref{E:BharMFrep},
is our base case.  
In Section \ref{S:chartable}, as a corollary of our main theorem, we see 
that any group formed from symmetric groups by taking wreath and cross products has a rational character table.  This
result, at least in such simple form, is not easily found in the literature.
In order to suggest what our results say in the language of field extensions,
in Section~\ref{S:towers} we mention the relationship between Galois representations
to wreath products and towers of field extensions.

In Section~\ref{S:further}, we discuss some situations
in which groups have mass formulas for one way of counting but not another.  
In particular, we show that $D_4\cong S_2\wr S_2$ does not have a mass formula for all local fields
when $c(\rho)$ is the counting function attached to 
the standard representation of $S_4$ restricted to $D_4\subset S_4.$
Consider quartic extensions $M$ of $K$, whose Galois closure has group $D_4$, with
quadratic subfield $L$.
The counting function that gives the mass formula for $D_4$ of Theorem~\ref{T:allMF}
corresponds to counting such extensions $M$ weighted by
$$|\Disc (L|K) N_{L|K} (\Disc (M|L))|^{-1},$$
whereas the counting function attached to 
the standard representation of $S_4$ restricted to $D_4\subset S_4$
corresponds to counting such extensions $M$ weighted by
$$|\Disc(M|K)|^{-1}= |\Disc (L|K)^2 N_{L|K} (\Disc (M|L))|^{-1}.$$
So this change of exponent in the $\Disc (L|K)$ factor affects the existence of a mass formula for all local fields.

\end{section}

\begin{notation}
Throughout this paper, $K$ is a local field and $G_K:=\Gal(K^{\mathrm{sep}}/K)$ is
the absolute Galois group of $K$.  
All maps in this paper from $G_K$ or subgroups of $G_K$ are continuous 
homomorphisms, with the discrete topology on all finite groups. 
 We let
 $I_K$ denote the inertia subgroup of $G_K$.  Recall that $S_{K,\Gamma}$ is the set of maps $G_K\ra \Gamma$, and
 $q_K$ is the size of residue field of $K$. 
Also, $\Gamma$ will always be a permutation group acting on a finite set.
\end{notation}

\begin{section}{Proof of Theorem~\ref{T:allMF}}\label{S:proof}

In order to prove Theorem~\ref{T:allMF}, we prove finer mass formulas first.
Instead of summing over all representations of $G_K$, we stratify the representations by \emph{type} and prove mass formulas 
for the sum of representations of each type.    Let
$\rho \colon G_K \ra \Gamma$ be a representation such that the action of $G_K$ has $r$ orbits $m_1,\dots,m_r$.  
If, under restriction to the 
representation $\rho : I_K \ra \Gamma$, orbit $m_i$ breaks up into $f_i$ orbits of size $e_i$, then we say
that $\rho$ is of \emph{type} $(f_1^{e_1} f_2^{e_2} \cdots f_r^{e_r})$ (where the terms $f_i^{e_i}$ are unordered formal symbols, as in \cite[Section 2]{Bhargava}).  
Let $L_i$ be the fixed field of the stabilizer of an element in $m_i$.
So, $[L_i\colon K]=|m_i|$. 
Since $I_{L_i}=G_{L_i}\cap I_K$ is the stabilizer in $I_K$ of an element in $m_i$, we conclude that $e_i=[I_K:I_{L_i}]$,
which is the ramification index of $L_i /K$.  Thus, $f_i$ is the inertial degree of $L_i /K$.

Given $\Gamma$, a counting function $c$ for $\Gamma$, and a type
$\sigma=(f_1^{e_1} f_2^{e_2} \cdots f_r^{e_r})$, we define the \emph{total mass} of $(K,\Gamma,c, \sigma)$ to be
$$
M(K,\Gamma,c, \sigma):=\sum_{\substack{\rho\in S_{K,\Gamma}\\ \mathrm{type}\  \sigma}}
 \frac{1}{q_K^{c(\rho)}}.
$$
We say that $(\Gamma,c)$ has \emph{mass formulas for $S$ by type} if for every type $\sigma$
 there exists a polynomial $f_{(\Gamma, c,\sigma)}(x)\in\Z[x]$ such
that for all local fields $K\in S$ we have
$$
M(K,\Gamma,c,\sigma)=f_{(\Gamma,c,\sigma)}\left(\frac{1}{q_K}\right).
$$
Bhargava \cite[Proposition 1]{Bhargava} actually proved that $S_n$ has mass formulas for all local fields
by type.  Of course, if $(\Gamma,c)$ has mass formulas by type, then we can sum over
all types to obtain a mass formula for $(\Gamma,c)$.

The key step in the proof of Theorem~\ref{T:allMF} is the following.

\begin{theorem}\label{T:MFR}
If $A$ and $B$ are finite permutation groups, 
 $S$ is some class of local fields, and
 $(A,c_A)$ and $(B,c_B)$ have mass formulas for $S$ by type, then there exists a counting function $c$ (given in Equation~\eqref{E:count})  such that
$(A\wr B, c)$ has  mass formulas for $S$ by type.
\end{theorem}

\begin{proof}
 Let $K$ be a local field in $S$.
Let $A$ act on the left on the set $\A$ and $B$ act on the left on the set $\B$.
We take the natural permutation action of $A \wr B$ acting
on a disjoint union of copies of $\A$ indexed by elements of $\B$.
Fix an ordering on $\B$ so that we have canonical orbit representatives in $\B$.
Given $\rho : G_K \ra A\wr B$, there is a natural quotient $\bar{\rho}: G_K \ra B$.  
Throughout this proof, we use $j$ as an indexing variable for the set $\B$
and $i$ as an indexing variable for the $r$ canonical orbit representatives in $\B$ of the $\rho(G_K)$ action.
Let $i_j$ be the index of the orbit representative of $j$'s orbit. 
Let $S_j\subset G_K$ be the stabilizer of $j$, and let $S_j$ have fixed field $L_j$.
We define $\rho_j \colon G_{L_j} \ra A$ to be the given action of $G_{L_j}$ on the $j$th copy of $\A$.  
We say that $\rho$ has \emph{wreath type}
\begin{equation}\label{E:wreathtype}
\Sigma =(f_1^{e_1}(\sigma_1)\cdots f_r^{e_r}(\sigma_r))
\end{equation}
if $\bar{\rho}$ has type $\sigma=(f_1^{e_1}\cdots f_r^{e_r})$ (where $f_i^{e_i}$ corresponds to the orbit of $i$) and
$\rho_i$ has type $\sigma_i$.  Note that type is a function of wreath type; if $\rho$ has wreath type $\Sigma$ as above where $\sigma_i=(f_{i,1}^{e_{i,1}}\cdots f_{i,r_i}^{e_{i,r_i}})$, then $\rho$ has type
$((f_i f_{i,k})^{e_ie_{i,k}})_{1\leq i\leq r,\, 1\leq k \leq r_i}$.

We consider the function $c$ defined as follows:
\begin{equation}\label{E:cdef}
c(\rho)=c_B(\bar{\rho})+\sum_{j\in \B} \frac{c_A(\rho_j)}{|\{\bar{\rho}(I_K) j\}|}.
\end{equation}
Since $c_B(\bar{\rho})$ only depends on the $B$-conjugacy class of $\bar{\rho}$
and $c_A(\rho_j)$ depends only on the $A$-conjugacy class of $\rho_j$,
we see that conjugation by elements of $A\wr B$ does not affect the right hand side of Equation~\eqref{E:count}
except by reordering the terms in the sum.  Thus $c$ is a counting function.

Since $\rho_j$ and $\rho_{i_j}$ 
are representations of conjugate subfields of $G_K$ and since $c_A$ is invariant under $A$-conjugation, $c_A(\rho_j)=
c_A(\rho_{i_j})$.  
There are $f_ie_i$ elements in the orbit of $i$ under $\bar{\rho}(G_K)$ 
and $e_{i_j}$ elements in the orbit of $j$ under $\bar{\rho}(I_K)$,
so $c(\rho)
= c_B(\bar{\rho})+\sum_{i=1}^r \frac{f_ie_i}{e_i} c_A(\rho_i)$ and thus
\begin{equation}\label{E:count}
c(\rho)= c_B(\bar{\rho})+\sum_{i=1}^r f_i c_A(\rho_i).
\end{equation}
Using this expression for $c(\rho)$, we will prove that $(A \wr B,c)$ has mass formulas by wreath type.  Then, summing over wreath types
that give the same type, we will prove that $(A \wr B,c)$ has mass formulas by type.

\begin{remark}
 For a permutation group $\Gamma$, let $d_\Gamma$ be the counting function attached to the permutation representation of $\Gamma$
(which is the discriminant exponent of the associated \'{e}tale extension).  Then we can compute
\begin{equation*}
 d_{A \wr B} = |\A| d_B(\bar{\rho})+\sum_{i=1}^r f_i d_A(\rho_i),
\end{equation*}
which is similar to the expression given in Equation~\eqref{E:count} but differs by the presence of $|\A|$ in the first term.
In particular, when we have mass formulas for $(A,d_A)$ and $(B,d_B)$, the mass formula for $A \wr B$ that we find in this paper is not with the counting function $d_{A \wr B}$.  We will see in Section~\ref{S:further}, when $A$ and $B$ are both $S_2$, that
$S_2 \wr S_2 \cong D_4$ does not have a mass formula with $d_{A \wr B}$.
\end{remark}

\begin{lemma}\label{L:manytoone}
The correspondence $\rho \mapsto (\bar{\rho}, \rho_1, \dots ,\rho_r)$ described above  
gives a function $\Psi$ from $S_{K,A\wr B}$ to tuples $(\phi,\phi_1,\dots,\phi_r)$ where
$\phi : G_K \ra B$, the groups $S_i$ are the stabilizers of canonical orbit representatives of the action of $\phi$ on $B$,
and $\phi_i \colon S_i \ra A$.
The map $\Psi$ is $(|A|^{|\B|-r})$-to-one and surjective.
\end{lemma}
\begin{proof}
Lemma~\ref{L:manytoone} holds when $G_K$ is replaced by any group.
It suffices to prove the lemma when $\bar{\rho}$ and $\phi$ are transitive because the general statement follows
by multiplication.  Let $b\in \B$ be the canonical orbit representative.  
Given a $\phi \colon G_K \ra B$ (or a 
 $\bar{\rho} \colon G_K \ra B$), for all $j\in\B$
choose a $\sigma_j \in G_K$ such that $\phi(\sigma_j)$ takes $b$ to $j$.
Given a $\rho \colon G_K \ra A \wr B$,
let $\alpha_j$ be the element of $A$ such that $\rho(\sigma_j)$ 
acts on the $b$th copy of $\A$ by $\alpha_j$ and then moves the $b$th copy of $\A$ to the $j$th copy.
Then for $g\in G_K$, the map $\rho$ is given by
\begin{equation}\label{E:rhodef}
\rho(g)=\bar{\rho}(g)(a_j)_{j\in\B}\in BA^{|\B|}=A \wr B, \quad \mbox{where  }  a_j=\alpha_{\bar{\rho}(g)(j)} \rho_1(\sigma_{\bar{\rho}(g)(j)}^{-1}g  \sigma_j) \alpha_j^{-1},
\end{equation} 
and $a_j\in A$ acts on the $j$th copy of $\A$.
For any transitive maps
$\phi \colon G_K \ra B$ and $\phi_b :\ S_b \ra A$ and  for any
choices of $\alpha_j\in A$ for all $j\in \B$ such that $\alpha_b=\phi_b(\sigma_b)$, we can check
that Equation~\eqref{E:rhodef} 
for $\bar{\rho}=\phi$ and $\rho_1=\phi_b$
gives a homomorphism $\rho \colon G_K \ra A \wr B$
with $(\bar{\rho},\rho_1)=(\phi,\phi_b)$, which proves the lemma.
\end{proof}

If $\Sigma$ is as in Equation~\eqref{E:wreathtype}, then
\begin{equation}\label{E:fullsum}
\sum_{\substack{\rho : G_K\ra A\wr B\\ \mathrm{wreath\,\, type}\ \Sigma}} \frac{1}{q_K^{c(\rho)}}=
|A|^{|\B|-r}\sum_{\substack{\phi : G_K\ra B\\ \mathrm{type}\ \sigma}}\sum_{\substack{\phi_1 :S_1 \ra A \\ \mathrm{type}\ \sigma_1}} 
\sum_{\substack{\phi_2 :S_2 \ra A \\ \mathrm{type}\ \sigma_2}} \cdots \sum_{\substack{\phi_r: S_r \ra A \\ \mathrm{type}\ \sigma_r}}
\frac{1}{q_K^{c_B(\phi)+\sum_{i=1}^r f_ic_A(\phi_i)}},
\end{equation}
where $S_i$ is the stabilizer under $\phi$ of a canonical orbit representative of
the action of $\phi$ on $\B$.
The right hand side of Equation~\eqref{E:fullsum} factors, and $S_i\subset G_K$ has fixed field $L_i$ with
 residue field of size $q_K^{f_i}$.
We conclude that
\begin{eqnarray*}
\sum_{\substack{\rho : G_K\ra A\wr B\\ \mathrm{wreath\,\, type}\ \Sigma}} \frac{1}{q_K^{c(\rho)}}&=&
|A|^{|\B|-r}\sum_{\substack{\phi : G_K\ra B\\ \mathrm{type}\ \sigma}} \frac{1}{q_K^{c_A(\phi)}}
\sum_{\substack{\phi_1 :G_{L_1} \ra A \\ \mathrm{type}\ \sigma_1}}  \frac{1}{q_K^{f_1{c_B(\phi_1)}}}
\cdots \sum_{\substack{\phi_r: G_{L_r} \ra A \\ \mathrm{type}\ \sigma_r}}
\frac{1}{q_K^{f_r{c_B(\phi_r)}}}\\
&=& |A|^{|\B|-r} f_{(B,c_B,\sigma)}\left(\frac{1}{q_K}\right)\prod_{i=1}^r f_{(A,c_A,\sigma_i)} \left(\frac{1}{q_K^{f_i}}\right).
\end{eqnarray*}
So, $(A\wr B,c)$ has mass formulas by wreath type, and thus by type.
\end{proof}

Kedlaya \cite[Lemma 2.6]{Kedlaya} noted that if $(\Gamma,c)$ and $(\Gamma',c')$ have mass formulas $f$ and $f'$, then $(\Gamma\times\Gamma',c'')$ has mass formula $ff'$, where $c''(\rho \times \rho')=c(\rho)+c'(\rho')$.
We can strengthen this statement to mass formulas by type using a much easier version of our argument for wreath products.
We define the \emph{product type} of a representation $\rho \times \rho' : G_K \ra \Gamma\times\Gamma'$
to be $(\sigma,\sigma')$, where $\sigma$ and $\sigma'$ are the types of $\rho$ and $\rho'$ respectively.  Then
$$
\sum_{\substack{\rho \times \rho' : G_K\ra \Gamma\times\Gamma'\\ \mathrm{product}\,\mathrm{ type}\ (\sigma,\sigma')}} \frac{1}{q_K^{c''(\rho \times \rho')}}=
\sum_{\substack{\phi : G_K\ra \Gamma\\ \mathrm{type}\ \sigma}} \frac{1}{q_K^{c(\rho)}}
\sum_{\substack{\phi_1 :G_{L_1} \ra \Gamma' \\ \mathrm{type}\ \sigma'}}  \frac{1}{q_K^{c'(\rho')}}.
$$
If $\Gamma$ and $\Gamma'$ have mass formulas by type, then the above gives
mass formulas of $\Gamma\times\Gamma'$ by product type.
Since type is a function of product type, we can sum the mass formulas
 by product type to obtain mass formulas by type
for $\Gamma\times\Gamma'$.
This, combined with Theorem~\ref{T:MFR} and
Bhargava's mass formula for $S_n$ by type \cite[Proposition 1]{Bhargava}, proves Theorem~\ref{T:allMF}.

\section{Groups with rational character tables}\label{S:chartable}
Kedlaya \cite[Proposition 5.3, Corollary 5.4, Corollary 5.5]{Kedlaya} showed that if $c(\rho)$ is the counting function attached to $\Gamma
\ra \GL_n(\C)$, then the following statement holds:
$(\Gamma,c)$ has a mass formula for all local fields $K$ with $q_K$ relatively prime to $|\Gamma|$ if and only if the character table of
$\Gamma$ has all rational entries.  
Kedlaya's proofs of \cite[Proposition 5.3, Corollary 5.4, Corollary 5.5]{Kedlaya}  hold for any counting function
$c$ that is determined by $\rho(I_K)$.
This suggests that we define a \emph{proper} counting function to be a counting function $c$ that satisfies the following: if we have $\rho:G_K\ra \Gamma$ and $\rho':G_{K'} \ra \Gamma$ with $q_{K},q_{K'}$ relatively prime to $|\Gamma|$, and if $\rho(I_{K})=\rho'(I_{K'})$, then 
$c(\rho)=c(\rho')$.

For proper counting functions, we always have partial mass formulas
proven as in \cite[Corollary 5.4]{Kedlaya}.
 \begin{proposition}
  Let $a$ be an invertible residue class mod $|\Gamma|$ and $c$ be a proper counting function.
Then $(\Gamma,c)$ has a mass formula for all local fields $K$ with $q_K\in a$.
 \end{proposition}
The following proposition says exactly when these partial mass formulas agree, again proven as in \cite[Corollary 5.5]{Kedlaya}.
\begin{proposition}\label{P:RatChar}
Let $c$ be a proper counting function for $\Gamma$.  Then $(\Gamma,c)$ has a mass formula for all local fields 
$K$ with $q_K$ relatively prime to $|\Gamma|$ if and only if
$\Gamma$ has a rational character table.
\end{proposition}
So, when looking for a group and a proper counting function with mass formulas for all local fields, we should look among groups with rational character tables (which are relatively rare; for example, including only 14 of the 93 groups of order $<32$ \cite{Conway}).
All specific counting functions that have been so far considered in the literature are proper.  It is not 
clear if there are any interesting non-proper counting functions.

Our proof of Theorem~\ref{T:MFR} has the following corollary.
\begin{corollary}\label{C:RatChar}
Any permutation group that can be constructed from the symmetric groups using wreath products and cross products
has a rational character table.
\end{corollary}
\begin{proof}
We first show that the counting function $c$ defined in Equation~\eqref{E:cdef} is proper if
$c_A$ and $c_B$ are proper.  We consider only fields $K$ with $q_K$ relatively prime to $|\Gamma|$.
Since $c_B(\bar{\rho})$ only depends on
$\bar{\rho}(I_K)$, it is clear that the $c_B(\bar{\rho})$ term only depends on $\rho(I_K)$.

Since $I_{L_j}=I_K \cap S_j$, we have $\rho_j(I_{L_j})=\rho(I_{L_j})=\rho(I_K)\cap\Stab(j)$.
  Since $c_A(\rho_j)$ 
depends only on $\rho_j(I_{L_j})$, we see that it depends only on $\rho(I_K)$.
The sum in Equation~\eqref{E:cdef}
then depends only on $\rho(I_K)$.
  So the $c$ defined in Equation~\eqref{E:cdef} is proper.
  Clearly the $c''(\rho\times\rho')$ defined for cross products is proper if $c$ and $c'$ are proper.
   The counting function in Bhargava's mass formula for $S_n$ (Equation~\eqref{E:BharMFrep})
       is an Artin conductor
  and thus is proper.  So we can prove Theorem~\ref{T:allMF} with a proper counting function and conclude the corollary.
  
 One can show in a similar way that even in wild characteristics, the 
  counting function $c$ defined in Equation~\eqref{E:count} depends only on the images  of
  the higher ramification groups $G^m_K$; i.e., if $\rho :  G_K \ra A \wr B$ and $\rho' :  G_K' \ra A \wr B$
  have $\rho(G^m_K)=\rho'(G^m_{K'})$ for all $m\in [0,\infty)$, then $c(\rho)=c(\rho'),$
  as long as the same is true for $c_A$ and $c_B$.
\end{proof}

So, for example, $((S_7\wr S_4)
 \times S_3) \wr S_8$ has a rational character table.
   Corollary~\ref{C:RatChar} does not seem to be a well-reported fact in the literature; 
   the corollary shows that all Sylow $2$-subgroups of symmetric groups
  (which are cross products of wreath products of $S_2$'s) have rational character table, which was posed
  as an open problem in \cite[Problem 15.25]{Kourovka} in 2002 and 
  solved in \cite{Revin} in 2004 and \cite{Kolesnikov} in 2005.
However, since $A\wr (B\wr C)=(A\wr B)\wr C$ and $A\wr(B \times C)=(A\wr B) \times (A\wr C)$,
any of the groups of Corollary~\ref{C:RatChar} can be constructed only using the cross product
and $\wr S_n$ operations.  
It is well known that the cross product of two groups with rational character tables has a rational character table.
Furthermore, Pfeiffer \cite{GAP} explains how GAP computes the character table of $G \wr S_n$ from
the character table of $G$, and one can check that if $G$ has rational character table then all of the values constructed in the the character table of $G \wr S_n$ are rational, which implies Corollary~\ref{C:RatChar}.

One might hope that all groups with rational character tables have mass formulas by type,
but this is not necessarily the case.   For example, considering $(C_3 \times C_3) \rtimes C_2$ (where $C_2$ acts non-trivially on each factor separately) in the tame case 
in type $(1^3\, 2^1 \, 1^1)$, one can check that for $q\equiv 1 \pmod{3}$ the mass is zero
and for $q\equiv 2 \pmod{3}$ the mass is non-zero.

\end{section}

\section{Towers and direct sums of field extensions}\label{S:towers}
Kedlaya explains the correspondence between Galois permutation representations and \'{e}tale extensions in
\cite[Lemma 3.1]{Kedlaya}.  We have seen this correspondence already in other terms.  If we have a representation
$\rho : G_K \ra \Gamma$ with $r$ orbits, $S_i$ is the stabilizer of an element in the $i$th orbit, and $L_i$ is the fixed field of $S_i$,
then $\rho$ corresponds to $L=\bigoplus_{i=1}^r L_i$. 
For a local field $F$, let $\wp_F$ be the prime of $F$.
 In this correspondence,
if $c$ is the counting function attached to the permutation representation of $\Gamma$, then $c$
is the discriminant exponent of the extension $L/K$ \cite[Lemma 3.4]{Kedlaya}.  In other words, $\wp_K^{c(\rho)}=\Disc(L|K)$.

We can interpret the representations $\rho :G_K \ra A \wr B$ as towers of \'{e}tale extensions $M/L/K$.  
If we take $\bar{\rho}: G_K \ra  B$, then $L=\bigoplus_{i=1}^r L_i$ is just the \'{e}tale extension of $K$ corresponding to 
$\bar{\rho}$. 
 Then if $M$ is the \'{e}tale extension of $K$ corresponding to $\rho$, we see that
$M=\bigoplus_{i=1}^r M_i$, where $M_i$ is the \'{e}tale extension of $L_i$ corresponding to
$\rho_i: G_{L_i} \ra A$.   So we see
that $M$ is an \'{e}tale extension of $L$, though $L$ might not be a field. 

Let $c$ be the counting function of our mass formula
for wreath products (given by Equation~\eqref{E:count}).  From Equation~\eqref{E:count}, we obtain 
\begin{equation*}
 \wp_K^{c(\rho)}=\wp_K^{c_B(\bar{\rho})} \prod_{i=1}^r N_{L_i|K} (\wp_{L_i}^{c_A(\rho_i)}).
\end{equation*}
For example, if $c_A$ and $c_B$ are both given by the discriminant exponent (or equivalently, attached to the permutation representation), then
\begin{equation}\label{E:fromdisc}
 \wp_K^{c(\rho)}=\Disc (L|K) \prod_{i=1}^r N_{L_i|K} (\Disc (M_i|L_i)).
\end{equation}
For comparison, $\Disc (M|K)=\Disc (L|K)^{[M:L]} \prod_{i=1}^r N_{L_i|K} (\Disc (M_i|L_i))$.

As we will see for $\Gamma=D_4$ in the next section, representations $\rho: G_K \ra \Gamma$ can give not only field extensions of $K$ whose Galois closure
has Galois group $\Gamma$, but also field extensions whose Galois closure
has Galois group a proper subgroup of $\Gamma$, as well as direct sums of field extensions.
One could say
that representations $\rho: G_K \ra A \wr B$ correspond to towers of ``$A$-extensions'' over ``$B$-extensions'' and further relate iterated wreath products to iterated towers.  
Similarly, one could say
that a representation $\rho: G_K \ra A \times B$ corresponds to a direct sum of an ``$A$-extension'' and a ``$B$-extension.''  
The quotes
indicate that the extensions do not necessarily have Galois closure with group $A$ or $B$.  In fact, it
seems the most convenient way to define ``$A$-extensions'' or isomorphisms of ``$A$-extensions'' is simply to use the language
of Galois representations as we have in this paper. 

\section{Masses for $D_4$}\label{S:further}
By Proposition~\ref{P:RatChar} we know, at least for proper counting functions,
that   the existence of a mass formula for a group $\Gamma$ for fields with $q_K$ relatively prime to $|\Gamma|$
 does not depend on the choice of the counting function.  However,
in wild characteristic this is not the case.  For example, $D_4$, the dihedral group with 8 elements,
is isomorphic to $S_2 \wr S_2$, so by Theorem~\ref{T:allMF} there is a $c$ (given in Equation~\eqref{E:count}) for which $D_4$ has a mass formula for all
local fields.  An expression for $c$ in terms of \'{e}tale extensions can be read off from Equation~\eqref{E:fromdisc}. In particular, 
for a 
surjective representation $\rho: G_K \ra D_4$ corresponding to a quartic field extension $M$ of $K$ with a quadratic subextension $L$,
\begin{equation}\label{E:Ex1}
 \wp_K^{c(\rho)}=\Disc (L|K) N_{L|K} (\Disc (M|L)).
\end{equation}
For this $c$, for all local fields $K$, we have that
$$M(K,D_4,c):=\sum_{\rho\in S_{K,D_4}} \frac{1}{q_K^{c(\rho)}}=8+\frac{16}{q_K}+\frac{16}{q_K^2}.$$
From
the definition of $c$ given in Equation~\eqref{E:cdef} and the description of the absolute tame Galois group of a local field,
we can compute $M(K,D_4,c)$ for a field $K$ with $q_K$ odd. 
By Theorem~\ref{T:MFR} we know the formula holds for all $K$.

However, the counting function for $D_4$ that has been considered when counting global extensions 
(for example in \cite{D4Q})
is the one attached the the faithful permutation representation of $D_4$ on a four element set (equivalently the discriminant exponent of the corresponding \'{e}tale extension). 
We call this counting function $d$, and in comparison with Equation~\eqref{E:Ex1} we have
\begin{equation*}
 \wp_K^{d(\rho)}=\Disc(M|K)=\Disc (L|K)^2 N_{L|K} (\Disc (M|L)).
\end{equation*}
With $d$, we now show that $D_4$ does
not have a mass formula for all local fields.

Using the correspondence of Section~\ref{S:towers}, we can analyze the representations $\rho: G_K \ra D_4\subset S_4$ in the following table,
where $I=\im(\rho)$ and $j=|\{s\in S_4|\, sIs^{-1}\subset D_4\}|$ and $k=|\operatorname{Centralizer}_{S_4}(I)|$.
 We take the
$D_4$ in $S_4$ generated by $(1\,2\,3\,4)$ and $(1\,3)$.

\begin{tabular}{l|l|l|l}
$I$ & $j$ & $k$ & $L$\\[2pt]
\hline
\hline
& & &\\[-8pt]
$D_4$ & 8 & 2 & degree 4 field whose Galois closure $/K$ has group $D_4$\\[4pt]
\hline
& & &\\[-8pt]
$C_4$ & 8 & 4 & degree 4 field Galois $/K$ with group $C_4\cong\Z/4$\\[4pt]
\hline
& & &\\[-8pt]
$\langle(1\, 2)(3\,4),(1\,3)(2\,4)\rangle$ & 24 & 4 & degree 4 field Galois $/K$ with group $V_4\cong\Z/2\times \Z/2$\\[4pt]
\hline
& & &\\[-8pt]
$\langle(1\, 3),(2\,4)\rangle$ & 8 & 4 & $L_1\oplus L_2$ with $[L_i:K]$=2 and $L_i$ distinct fields\\[4pt]
\hline
& & &\\[-8pt]
\parbox{1.7in}{$\langle(1\,3)(2\,4)\rangle$, $\langle(1\,2)(3\,4)\rangle$, or $\langle(1\,4)(2\,3)\rangle$ }& 24 & 8 & $L_1\oplus L_2$ with $[L_i:K]$=2, and $L_1\cong L_2$ fields \\[12pt]
\hline
& & &\\[-8pt]
$\langle(2\,4)\rangle$ or $\langle(1\,3)\rangle$ & 8 & 4 & $L_1\oplus K\oplus K$ with $[L_1:K]$=2, and $L_1$ a field\\[4pt]
\hline
& & &\\[-8pt]
$1$  & 24 & 24 & $K\oplus K\oplus K\oplus K$ \\[4pt]
\end{tabular}

Each isomorphism class of algebras appears $\frac{j}{k}$ times from a representation $\rho:G_{K}\ra D_4$ (see \cite[Lemma 3.1]{Kedlaya}).
Let $S(K,G,m)$ be the set of isomorphism classes of degree $m$ field extensions of $K$ whose Galois closure over $K$ 
has group $G$.  Then from the above table we see that
\begin{eqnarray*}
M(K,D_4,d)&=&\sum_{F\in S(K,D_4,4)} \frac{4}{|\Disc F|} +\sum_{F\in S(K,C_4,4)} \frac{2}{|\Disc F|}+\sum_{F\in S(K,V_4,4)} \frac{6}{|\Disc F|}\\
&&+\sum_{\substack{F_1,F_2\in S(K,C_2,2)\\F_1\not\cong F_2}} \frac{2}{|\Disc F_1||\Disc F_2|}
+\sum_{F\in S(K,C_2,2)} \frac{3}{|\Disc F|^2}\\
&&+\sum_{F\in S(K,C_2,2)} \frac{2}{|\Disc F|}+1.
\end{eqnarray*}
where if $\wp_F$ is the prime of $F$ and $\Disc F=\wp_F^m$, then $|\Disc F|=q_F^m$.
Using the Database of Local Fields \cite{DataLF} we can compute that $M(\Q_2,D_4,d)=\frac{121}{8}$.  For
fields with $2\nmid q_K$, the structure of the tame quotient of the absolute Galois group of a local field allows us to compute the mass
 to be $8+ \frac{8}{q_K}+ \frac{16}{q_K^2}+ \frac{8}{q_K^3}$ (also see \cite[Corollary 5.4]{Kedlaya}) which evaluates
to 17 for $q_K=2$.  Thus $(D_4,d)$ does not have a mass formula for all local fields.

As another example, Kedlaya \cite[Proposition 9.3]{Kedlaya} found that $W(G_2)$ does not have a mass formula for all local fields of residual characteristic 
$2$ when $c$ is the Artin conductor of the Weyl representation.  However, $W(G_2)\cong S_2 \times S_3$ and thus it has a mass formula
for all local fields with counting function the sum of the Artin conductors of the standard representations of $S_2$ and $S_3$.  

It would be interesting to study what the presence or absence of mass formulas tells us about a counting function, in particular with respect
to how global fields can be counted asymptotically with that counting function.
As in Bhargava's work \cite[Section 8.2]{Bhargava}, we can form an Euler series
\begin{equation*}
 M_c(\Gamma,s)=C(\Gamma)\left(\sum_{\rho\in S_{\R,\Gamma}}\frac{1}{|\Gamma|} \right) \prod_p \left(\frac{1}{|\Gamma|}\sum_{\rho\in S_{\Q_p,\Gamma}} 
\frac{1}{p^{c(\rho)s}}\right)=\sum_{n\geq 1} m_n n^{-s},
\end{equation*}
where $C(\Gamma)$ is some simple, yet to be explained, rational constant.
(We work over $\Q$ for simplicity, and the product is over rational primes.)
For a representation $\rho: G_\Q \ra \Gamma$, let $\rho_p$ be the restriction of
$\rho$ to $G_{\Q_p}$.  The idea is that
$m_n$ should be a heuristic
of the number of $\Gamma$-extensions of $\Q$ 
(i.e. surjective $\rho : G_\Q \ra \Gamma$)
with $\prod_p p^{c(\rho_p)}=n$, though
$m_n$ is not necessarily an integer.

Bhargava \cite[Section 8.2]{Bhargava} asks the following.
\begin{question}\label{Q}
Does
$$
\lim_{X\ra \infty} \frac{\sum_{n=1}^X m_n}{|\{\textrm{isom. classes of surjective }\rho : G_\Q \ra \Gamma
\textrm{ with } \prod_p p^{c(\rho_p)}\leq X \}| } =1?
$$
\end{question}
Bhargava in fact asks more refined questions in which some local behaviors are fixed.
With the counting function $d$ for $D_4$  attached to the permutation representation (i.e. the discriminant exponent),
we can form $M_d(D_4,s)$ and compute numerically the above limit.
We use the work of Cohen, Diaz y Diaz, and Oliver on counting $D_4$-extensions by discriminant 
(see \cite{D4} for a recent value of the relevant constants) to calculate the limit of the denominator, and we use standard Tauberian
theorems (see \cite[Corollary, p. 121]{N}) and 
PARI/GP \cite{PARI} to calculate the limit of the numerator.
  Of course,
$C(D_4)$ has not been decided, but it does not appear (by using the {\tt algdep} function in PARI/GP) that any simple rational
$C(D_4)$ will give an affirmative answer to the above question

In light of our mass formula for a different counting function $c$ for $D_4$, we naturally wonder about Question~\ref{Q}
in the case of $D_4$ and that $c$.  Answering this question would require counting $D_4$ extensions $M$ with quadratic subfield $L$
by
$\Disc (L|\Q) N_{L|\Q} (\Disc (M|L))$ instead of by discriminant
(which is $\Disc (L|\Q)^2 N_{L|\Q} (\Disc (M|L))$).

\section*{Acknowledgements}
I would like to thank Manjul Bhargava for his guidance while I was doing this research and thorough and helpful comments on earlier versions of this manuscript.  I would also like to thank Kiran Kedlaya for providing me an early version of \cite{Kedlaya} and for answers to my questions about that paper.  The referee gave suggestions for additions and improvements to the paper that 
were incorporated and much appreciated.


\begin{thebibliography}{99}

\bibitem{Bh1} Bhargava, M., The density of discriminants of quartic rings and
fields, {\it Ann.\ of Math.} {\bf 162} (2005), 1031--1063.

\bibitem{Bh2} Bhargava, M., The density of discriminants of quintic rings and
fields, {\it Ann.\ of Math.}, to appear.

\bibitem{Bhargava} Bhargava, M., Mass formulae for extensions of local fields, and conjectures on the density of number field discriminants, \textit{Int. Math. Res. Not.}, to appear.

\bibitem{D4}
Cohen, H., F. Diaz y Diaz, and M. Olivier, Counting discriminants of number fields.  {\it J. Théor. Nombres Bordeaux} {\bf 18}  (2006),  no. 3, 573--593.

\bibitem{D4Q}
Cohen, H., F. Diaz y Diaz, and M. Olivier, Enumerating quartic dihedral extensions of $\Q$.  {\it Compositio Math.}  {\bf 133}  (2002),  no. 1, 65--93.

\bibitem{Conway} Conway, J.\ H., personal communication.

\bibitem{DH} Davenport, H. and H.\ Heilbronn, On the density of discriminants of
cubic fields II, {\it Proc.\ Roy.\ Soc.\ London Ser.\ A} {\bf 322} (1971),
no.\ 1551, 405--420.

\bibitem{Ellenberg}  Ellenberg, J. and A. Venkatesh,  Counting extensions of function fields with specified Galois group and bounded discriminant, \textit{Geometric Methods in Algebra and Number Theory}, F. Bogomolov and Y. Tschinkel, eds. (2005), 151--168.

\bibitem{Kedlaya} Kedlaya, K.\ S., 
Mass formulas for local Galois representations (with an appendix by Daniel Gulotta), \textit{Int. Math. Res. Not.}, to appear. 

\bibitem{Kourovka} Khukhro, E.\ I. and V.\ D. Mazurov, Eds., \textit{The Kourovka Notebook.  Unsolved Problems in group theory.  Fifteenth
augmented edition.} Novosibirsk: Russian Academy of Sciences, Siberian Branch, Institute of Mathematics (1999).

\bibitem{Kolesnikov} Kolesnikov, S.\ G., On the rationality and strong reality of Sylow 2-subgroups of Weyl and alternating groups, \textit{Algebra Logika}  {\bf 44}  (2005),  no. 1, 44--53, 127;  English translation in  \textit{Algebra Logic}  {\bf 44}  (2005),  no. 1, 25--30.

\bibitem{DataLF}
Jones, J.\ W.  and D.\ P. Roberts,
A database of local fields,
{\it J.\ Symbolic Comput.} {\bf 41} (2006), no.\ 1, 80--97.
{\tt \verb+http://math.asu.edu/~jj/localfields/+}


\bibitem{N}
Narkiewicz, W., translated by S. Kanemitsu,
\textit{Number Theory,} 
 World Scientific,  1983.

\bibitem{PARI}
PARI/GP, version {\tt 2.3.2}, Bordeaux, 2006, {\tt http://pari.math.u-bordeaux.fr/}.

\bibitem{GAP}
Pfeiffer, G., 
Character tables of Weyl groups in GAP,
\textit{Bayreuth. Math. Schr.} {\bf 47} (1994), 165--222.

\bibitem{Revin}  Revin, D.\ O., The characters of groups of type $X\wr \mathbb Z\sb p$, \textit{ Sib. \~{A}lektron. Mat. Izv.}  {\bf 1}  (2004), 110--116.

\bibitem{Serre} Serre, J.-P., Une ``formule de masse'' pour les extensions totalement ramifi\'{e}es de degr\'{e} donn\'{e}
d'un corps local, \textit{C. R. Acad. Sci. Paris.} S\'{e}r. A-B {\bf 286} (1978), no. 22, A1031--A1036.


\bibitem{Abelian} Wood, M.\ M., On the probabilities of local behaviors in abelian field extensions, in preparation.

\end{thebibliography}
\end{document}